\documentclass[a4paper,11pt]{amsart}
\usepackage{amsmath,amssymb,amsfonts}
\usepackage{epsfig}
\usepackage{mathrsfs}
\usepackage{ae}
\usepackage[utf8]{inputenc}

\theoremstyle{definition}
\newtheorem{defn}{Definition}[section]

\theoremstyle{plain}
\newtheorem{thm}{Theorem}[section]
\newtheorem{prop}{Proposition}[section]
\newtheorem{cor}{Corollary}[section]
\newtheorem{lma}{Lemma}[section]
\theoremstyle{remark}
\newtheorem{rmk}{Remark}[section]
\newtheorem{exm}{Example}[section]

\newcommand{\R}{\mathbb{R}}
\newcommand{\E}{\mathbb{E}}
\renewcommand{\P}{\mathbb{P}}

\newcommand{\ud}{\mathrm{d}}

\newcommand{\la}{\langle}
\newcommand{\ra}{\rangle}

\numberwithin{equation}{section}

\font\eka=cmex10

\def\ind{\mathrel{\hbox{\rlap{%
\hbox to 7.5pt{\hrulefill}}\raise6.6pt\hbox{\eka\char'167}}}}

\begin{document}
\title[Gaussian Pathwise It\^o--Tanaka Formula]{Pathwise integrals and It\^o--Tanaka Formula for Gaussian processes}

\author[Sottinen]{Tommi Sottinen}
\address{Department of Mathematics and Statistics, University of Vaasa, P.O. Box 700, FIN-65101 Vaasa, FINLAND}

\author[Viitasaari]{Lauri Viitasaari}
\address{Department of Mathematics and System Analysis, Aalto University School of Science, Helsinki\\
P.O. Box 11100, FIN-00076 Aalto,  FINLAND} 

\thanks{L. Viitasaari was financed by the Finnish Doctoral Programme in Stochastics and Statistics.}

\begin{abstract}
We prove an It\^o--Tanaka formula and existence of pathwise stochastic integrals for a wide class of Gaussian processes.  Motivated by financial applications, we define the stochastic integrals as forward-type pathwise integrals introduced by Föllmer and as pathwise generalized Lebesgue--Stieltjes integrals introduced by Zähle. As an application, we illustrate importance of It\^o--Tanaka formula for pricing and hedging of financial derivatives.
\end{abstract}

\keywords{Föllmer integral;
Gaussian processes;
generalized Lebesgue--Stieltjes integral;
It\^o--Tanaka formula; 
mathematical finance;
pathwise stochastic integral.}

\subjclass[2010]{60G15; 60H05; 91G20}   

\maketitle


\section{Introduction}

Let $Y$ be a continuous Gaussian stochastic process on a time interval $[0,T]$ and let $f$ be a linear combination of convex functions. We are interested in which generality and for what kind of integrals and notion of local time $L_T^a$ we can obtain the It\^o--Tanaka formula
\begin{equation}\label{eq:ito-tanaka}
f(Y_T) = f(Y_0) + \int_0^T f'_-(Y_u)\, \ud Y_u + 
\frac12 \int_{-\infty}^\infty L_T^a(Y)\, f''(\ud a).
\end{equation}  
Since we do not assume that $Y$ is a semimartingale, standard stochastic integration theory cannot be applied here and we have to  determine in which sense the stochastic integral in (\ref{eq:ito-tanaka}) exists. 

Motivated by financial applications we consider pathwise integrals that are generalizations of the financially natural Riemann--Stieltjes integral.  These generalizations go back to Young \cite{y}. In particular, we consider the pathwise forward-type Riemann--Stieltjes integral introduced by Föllmer \cite{Follmer} and the pathwise generalized Lebesgue--Stieltjes integrals introduce by Zähle \cite{z} and further studied e.g. by Nualart and R\u{a}\c{s}canu \cite{n-r}. 

Let us note that for Gaussian processes one can also consider Skorokhod integrals and e.g. in the particular case of fractional Brownian motion the It\^o--Tanaka formula (\ref{eq:ito-tanaka}) is established in \cite{c-n-t}. However, the Skorokhod integrals do not admit an economical interpretation in any obvious way; see \cite{s-v} for details.

Our work is related to \cite{Azmoodeh,a-m-v,a-v,n-r}, where only the case of fractional Brownian motion was studied. We extend these results to a more general class of Gaussian processes. 
Bertoin \cite{Bertoin} also established the It\^o--Tanaka formula (\ref{eq:ito-tanaka}) in 
the Föllmer sense for a very general class of Dirichlet process. 
In that sense our work is merely a special case of \cite{Bertoin}.  
However, in \cite{Bertoin} it was a priori assumed that the local time $L_T^a(Y)$ exists 
as the Lebesgue density of the occupation measure in the $L^2$ sense.  This actually also implies the existence of the associated 
F\"{o}llmer integral. We do not assume the existence 
of the density or the integral a priori.  Also, the generalized Lebesgue--Stieltjes integrals were not 
studied in \cite{Bertoin}.  

The rest of the paper is organized as follows. In section \ref{sec:aux} we introduce generalized 
Lebesgue-Stieltjes integrals and Föllmer integrals. The main section \ref{sec:main} begins with introducing our 
assumptions with a discussion and examples. Then we prove several fundamental lemmas after which we state and prove our 
main results on the existence of the generalized Lebesgue--Stieltjes integrals, the Föllmer integrals and mixed 
Föllmer--generalized Lebesgue--Stieltjes integrals. Then we prove the It\^o--Tanaka formula.  In Section \ref{sect:finance} we discuss the importance of the It\^o--Tanaka formula for the hedging and pricing of financial derivatives.  Finally, a technical lemma on level-crossing probabilities of Gaussian processes is given in the appendix. 

\section{Pathwise Integrals} 
\label{sec:aux}

We recall two notions of pathwise stochastic integrals: the generalized Lebesgue--Stieltjes integral and the Föllmer integral.

\subsection*{Generalized Lebesgue--Stieltjes Integral}

We first recall definitions for fractional Besov norms and Lebesgue--Liouville fractional integrals and derivatives. For details on fractional integration we refer to \cite{s-k-m} and for fractional Besov spaces we refer to \cite{n-r}.

\begin{defn}
Fix $ 0 <\beta < 1 $.
\begin{enumerate}
\item
The \emph{fractional Besov space} $W^{\beta}_1 =  W^{\beta}_1 ([0,T])$ is the space of real-valued measurable functions $ f :[0,T] \to \mathbb{R}$ such that
\begin{equation*}
{\Vert f \Vert}_{1,\beta} = \sup _{0 \le s < t \le T} \left( \frac{|f(t) - f(s)|}{(t-s)^\beta} + \int _{s}^{t} \frac{|f(u) - f(s) |}{(u-s)^ {1+\beta }} \,\ud u \right) < \infty.
\end{equation*}
\item
The \emph{fractional Besov space} $W^{\beta}_2 =  W^{\beta}_2 ([0,T])$ is the space of real-valued measurable functions  $ f :[0,T] \to \mathbb{R}$ such that
\begin{equation*}
{\Vert f \Vert}_{2,\beta} = \int_{0}^{T} \frac{|f(s)|}{s^ \beta}\, \ud s + \int_{0}^{T}\int_{0}^{s} \frac{|f(u) - f(s) |}{(u-s)^ {1+\beta }} \,\ud u \ud s < \infty.
\end{equation*}
\end{enumerate}
\end{defn}

\begin{rmk}\label{r:rmk1}
Let $C^{\alpha}=C^{\alpha }([0,T])$ denote the space of H\"{o}lder continuous functions of order $\alpha$ on $[0,T]$ and let $ 0< \epsilon < \beta \wedge (1- \beta)$. Then 
\begin{center}
$C^{\beta + \epsilon} \subset W^{\beta}_{1} \subset C^{\beta} 
\quad\mbox{and}\quad C^{\beta + \epsilon} \subset W^{\beta}_{2}$.
\end{center}
\end{rmk}

\begin{defn}
Let $t\in[0,T]$. The \emph{Riemann--Liouville fractional integrals} $I^\beta _{0+}$ and $I^\beta_{t-}$ of order $\beta > 0$ on $[0,T]$ are
\begin{eqnarray*}
(I^\beta _{0+} f)(s) &=& \frac{1}{\Gamma (\beta)} \int _0^s f(u) (s-u)^{\beta -1} \,\ud u, \\
(I^\beta _{t-} f)(s) &=& \frac{(-1)^{-\beta}}{\Gamma (\beta)} \int_s^t f(u) (u-s)^{\beta -1} \,\ud u, 
\end{eqnarray*}
where $\Gamma$ is the Gamma-function.  The \emph{Riemann--Liouville fractional derivatives} $D^{\beta}_{0+}$ and $D^{\beta}_{t-}$ are the left-inverses of corresponding integrals $I^\beta_{0+}$ and $I^\beta_{t-}$. They can be also defined via \emph{Weyl representation} as
\begin{eqnarray*}
(D^{\beta}_{0+} f)(s) &=& 
\frac{1}{\Gamma(1-\beta)} \left( \frac{f(s)}{s^\beta} + \beta \int_{0}^{s}\frac{f(s) - f(u)}{(s-u)^{\beta + 1}}\,\ud u \right),\\
(D^{\beta}_{t-} f)(s) &=& 
\frac{(-1)^{-\beta}}{\Gamma(1-\beta)} \left( \frac{f(s)}{(t-s)^\beta} + \beta \int_{s}^{t}\frac{f(s) - f(u)}{(u-s)^{\beta + 1}} \,\ud u \right)
\end{eqnarray*}
if $f\in I^\beta_{0+}(L^1)$ or $f\in I^\beta_{t-}(L^1)$, respectively.
\end{defn}

Denote $g_{t-}(s) = g(s)-g(t-)$. 

The generalized Lebesgue--Stieltjes integral is defined in terms of fractional derivative operators according to the next proposition.

\begin{prop}\cite{n-r}\label{pr:n-r}
Let $0<\beta<1$ and let $f \in  W^{\beta}_2$ and $ g \in W^{1- \beta}_1$. Then for any $t \in(0,T]$ the \emph{generalized Lebesgue--Stieltjes integral} exists as the following Lebesgue integral
$$
\int_0^t f(s)\, \ud g(s) =
(-1)^{\beta}\int_{0}^{t} (D^{\beta}_{0+} f)(s) (D^{1- \beta}_{t-} g_{t-} )(s) \,\ud s
$$
and is independent of $\beta$.
\end{prop}

\begin{rmk}\label{rmk:coincide}
It is shown in \cite{z} that if $f \in C^{\gamma}$ and $g \in C^{\eta}$ with $ \gamma + \eta > 1$, then the generalized Lebesgue--Stieltjes integral $ \int_{0}^{t} f(s) \,\ud g(s)$ exists and coincides with the classical Riemann--Stieltjes integral, i.e., as a limit of Riemann--Stieltjes sums.  This is natural, since in this case one can also define the integrals as \emph{Young integrals} \cite{y}. 
\end{rmk}

We will also need the following estimate in order to prove our main theorems.

\begin{thm}\cite{n-r}\label{t:n-r}
Let $ f \in  W^{\beta}_2$ and  $ g \in W^{1- \beta}_1$. Then we have the estimate
$$
\left|\int_{0}^t f(s) \,\ud g(s) \right| \le \sup_{0\le s < t \le T} \big| D^{1 - \beta}_{t-} g_{t-}(s) \big| {\Vert f \Vert}_{2,\beta}.
$$
\end{thm}

\subsection*{Föllmer integral}

We recall the definition of a forward-type Riemann--Stieltjes integral due to F\"{o}llmer \cite{Follmer} (see \cite{Sondermann} for English translation) and discuss its connection to the generalized Lebesgue--Stieltjes integral of Proposition \ref{pr:n-r}.
 
\begin{defn}\label{defn:follmer-integral}
Let $(\pi_n)_{n=1}^{\infty}$ be a sequence of partitions
$\pi_n=\{0=t_0^n<\ldots<t_{k(n)}^n=T\}$ such that $|\pi_n|=\max_{j=1,\ldots,k(n)}|t_j^n-t_{j-1}^n|\rightarrow 0$ as $n\to\infty$. Let $X$ be a continuous process. The \emph{Föllmer integral along the sequence $(\pi_n)_{n=1}^{\infty}$} of $Y$ with respect to $X$ is defined as 
$$
\int_0^t Y_u \,\ud X_u = \lim_{n\rightarrow\infty} \sum_{t_j^n\in\pi_n \cap(0,t]}Y_{t_{j-1}^n}(X_{t_j^n}-X_{t_{j-1}^n}),
$$
if the limit exists almost surely.  
\end{defn}

\begin{rmk}
\begin{enumerate}
\item
The Föllmer integral is  a \emph{forward-type} Riemann--Stieltjes integral.  Thus, if the Riemann--Stieltjes integral exists, so does the Föllmer integral, but not necessarily vice versa.  Also, it is clear that the Föllmer integral is a pathwise generalization of the It\^o integral, if one takes the sequence of partitions $(\pi_n)_{n=1}^\infty$ to be refining.    
\item
The Föllmer integral is a natural notion of integration in mathematical finance.  Indeed, the budget constraint of a self-financing trading strategy translates in the limit as a Föllmer integral.  See \cite{b-s-v-0,b-s-v} for further discussion.
\end{enumerate}
\end{rmk}

In general, it is very difficult to prove the existence of Föllmer integral.  In the case of so-called quadratic variation processes the existence is guaranteed by the It\^o--Föllmer formula of Lemma \ref{ito-follmer} below, which shows that the Föllmer integral behaves like the It\^o integral in the case of integrators with quadratic variation.

\begin{defn}
Let $(\pi_n)_{n=1}^{\infty}$ be a sequence of partitions $\pi_n=\{0=t_0^n<\ldots<t_{k(n)}^n=T\}$ such that $|\pi_n|=\max_{j=1,\ldots,k(n)}|t_j^n-t_{j-1}^n|\rightarrow 0$ as $n\to\infty$. Let $X$ be a continuous process. Then $X$ is a \emph{quadratic variation process along the sequence $(\pi_n)_{n=1}^\infty$} if the limit
$$
{\la X \ra}_t = \lim_{n\rightarrow\infty} \sum_{t_j^n\in\pi_n \cap(0,t]}\left(X_{t_j^n}-X_{t_{j-1}^n}\right)^2
$$
exists almost surely.
\end{defn}

\begin{rmk}
\begin{enumerate}
\item
For a standard Brownian motion $W$ we have $\ud\langle W \rangle_t = \ud t$ if the sequence $(\pi_n)$ is refining. This follows from the Borel-Cantelli lemma. 
\item
For continuous martingales $M$ their bracket $\la M\ra$ is also their quadratic variation (in the pathwise, Föllmer, sense) for suitably chosen sequences $(\pi_n)$. 
\item
If $Z$ is a continuous process with zero quadratic variation along $(\pi_n)$ and $X$ is a continuous quadratic variation process along $(\pi_n)$ then $\ud\langle X+Z\rangle_t = \ud\langle X\rangle_t$. This follows from the Cauchy-Schwartz inequality.
\item
If $X$ is a quadratic variation process along $(\pi_n)$ and $f\in C^1$ then $Y=f\circ X$ is also a quadratic variation process along $(\pi_n)$. Indeed,
$$
\ud\la Y \ra_t = f'(X_t)\,\ud\la X \ra_t.
$$
\end{enumerate}
\end{rmk}

\begin{lma}\label{ito-follmer}\cite{Follmer}
Let $X$ be a continuous quadratic variation process and let $f\in C^{1,2}([0,T]\times\R)$. Let $0\le s< t\le T$. Then
\begin{eqnarray*}
f(t,X_t) &=& f(s,X_s) + \int_s^t \frac{\partial f}{\partial t}(u,X_u)\, \ud u + \int_s^t \frac{\partial f}{\partial x}(u,X_u)\, \ud X_u \\
& &+ \frac{1}{2}\int_s^t \frac{\partial^2 f}{\partial x^2}(u,X_u)\, \ud\langle X\rangle_u.  
\end{eqnarray*}
In particular, the Föllmer integral exists and has a continuous modification.
\end{lma}

\section{Main Results}
\label{sec:main}

\subsection*{Notations, Definitions and Auxiliary Results}

\begin{defn}
Let $X$ be a centered Gaussian process. We denote by $R(t,s)$, $W(t,s)$, and $V(t)$ its covariance, incremental variance and variance, i.e.
\begin{eqnarray*}
R(t,s) &=& \E[X_tX_s], \\
W(t,s) &=& \E[(X_t - X_s)^2], \\
V(t)   &=& \E[X_t^2].
\end{eqnarray*}
We denote by $w^*(t)$ the ''worst case'' incremental variance
$$
w^*(t) = \sup_{0\le s\le T-t} W(t+s,s).
$$
\end{defn}

We begin with the following technical Lemma.
\begin{lma}
\label{lma:technical}
Let $R$ be a covariance of a centered process with $R(s,t)>0$. Let $0<s\leq t\leq T$. 
\begin{enumerate}
\item
If $R(s,s)\leq R(s,t)$, then
$$
1 - \frac{R(s,s)}{R(t,s)}\leq  \frac{\sqrt{W(t,s)}}{\sqrt{V(s)}},
$$
\item
if $R(s,s)> R(s,t)$, then
$$
\frac{R(s,s)}{R(t,s)}-1\leq  \frac{\sqrt{W(t,s)}}{\sqrt{V(s)}}\frac{R(s,s)}{R(t,s)}.
$$
\end{enumerate}
\end{lma}

\begin{proof}
For the claim (i), note that we have always $R(t,s)^2 \leq R(t,t)R(s,s)$. Hence
\begin{equation}
\label{eq_selvio}
\frac{R(t,s)^2}{R(s,s)} + R(s,s)-2R(t,s) \leq R(t,t) + R(s,s) -2R(t,s).
\end{equation}
Clearly we have 
$$
\frac{R(t,s)^2}{R(s,s)} + R(s,s)-2R(t,s) = \left(\frac{R(t,s)}{\sqrt{R(s,s)}}-\sqrt{R(s,s)}\right)^2
$$
and
$$
R(t,t) + R(s,s) -2R(t,s) = W(t,s).
$$
Hence by taking square root on both sides of (\ref{eq_selvio}) we obtain
\begin{equation*}
\frac{R(t,s)}{\sqrt{R(s,s)}}-\sqrt{R(s,s)}\leq \sqrt{W(t,s)}.
\end{equation*}
It remains to note that
$$
1-\frac{R(s,s)}{R(t,s)}\leq \frac{R(t,s)}{R(s,s)}-1.
$$

For the claim (ii), arguing as in the the proof of claim (i) above, we obtain that
\begin{equation*}
1-\frac{R(t,s)}{R(s,s)}\leq \frac{\sqrt{W(t,s)}}{\sqrt{V(s)}}.
\end{equation*}
Multiplying by $\frac{R(s,s)}{R(t,s)}$ gives the result.
\end{proof}

The key lemma to our analysis is the following estimate for the probability that the process $X$ crosses a fixed level:  $\P(X_s < a < X_t)$. Depending on values of $t$, $s$ and $a$ one can obtain different estimates and detailed bounds are given in Lemma \ref{lma:crossing} in the appendix. For our purposes we need the following universal estimate which holds for every value $s$, $t$ and $a$.

\begin{lma}
\label{lma:crossing2}
Let $X$ be a centered Gaussian process with strictly positive and bounded covariance function $R$, $0 < s < t \le T$ 
and $ a \in \R$. Then there exists a universal constant $C$ such that
\begin{eqnarray*}
\lefteqn{\P \big( X_s < a < X_t\big)} \\
&\le& C \frac{\sqrt{W(t,s)}}{\sqrt{V(s)}} \left[1+\frac{R(s,s)}{R(t,s)}+ \frac{|a|e^{-\frac{a^2}{2V^*}}}{\sqrt{V(s)}}\max\left(1,\frac{R(s,s)}{R(t,s)}\right)\right],
\end{eqnarray*}
where 
\begin{equation*}
V^* = \sup_{s\leq T}V(s).
\end{equation*}
\end{lma}
\begin{proof}
The claim follows directly by applying Lemma \ref{lma:technical} and inequality
$\sigma^2 \le W(t,s)$ on terms in Lemma \ref{lma:crossing}.
\end{proof}

Recall that a process  $X=(X_t)_{t\in[0,T]}$ is \emph{Hölder continuous of order $\alpha$} if there exists an almost surely finite random variable $C_T$ such that
\begin{equation*}
|X_t -X_s| \leq C_T|t-s|^{\alpha}
\end{equation*}
almost surely for all $s,t\in[0,T]$.

\begin{defn}\label{defn:Xalpha}
Let $0<\alpha<1$. A centered continuous Gaussian process $X=(X_t)_{t\in[0,T]}$ with covariance $R$ belongs to \emph{class $\mathcal{X}^\alpha$} if 
\begin{enumerate}
\item
$R(s,t)> 0$ for every $s,t>0$,
\item
the ''worst case'' incremental variance satisfies, at $t=0$,
\begin{equation*}
w^*(t) = Ct^{2\alpha} + o(t^{2\alpha}),
\end{equation*}
where $C> 0$,
\item
there exist $c,\delta>0$ such that
$$
V(s) \geq cs^{2},
$$
when $s\leq \delta$,
\item
there exists a $\delta>0$ such that
$$
\sup_{0< t<2\delta}\sup_{\frac{t}{2}\leq s\leq t}\frac{R(s,s)}{R(t,s)}< \infty.
$$
\end{enumerate}
\end{defn}

Definition \ref{defn:Xalpha} of the class $\mathcal{X}^\alpha$ is rather technical. However, next remarks and examples should convince the reader that assumptions are relatively natural and that the class $\mathcal{X}^\alpha$ is quite large.

\begin{rmk} 
\begin{enumerate}
\item
The first condition on strictly positive autocorrelation $R(t,s)> 0$ is rather restrictive.  
However, it seems that most Gaussian models do indeed satisfy it.  
Also, if one uses the generalized Lebesgue--Stieltjes approach, one has to assume $R(s,t)>0$. Indeed, otherwise our integrands 
would not belong to the required fractional Besov spaces.
\item
The second condition on the ``worst case'' incremental variance is the most important assumption: it implies that $X$ has a version which is  H\"{o}lder continuous of order $r$ for any $r<\alpha$ on $[0,T]$. Indeed, now
$$
\E\left[(X_t - X_s)^p\right] \leq C_p^p |t-s|^{p\alpha}
$$
for every $p\geq 1$. Hence the H\"{o}lder continuity follows directly from the Kolmogorov continuity theorem. 
\item
The third condition implies that the process is not too smooth.  Indeed, if the variance $V(s)$  behaves like $s^\gamma$  for some $\gamma\geq 2$ near zero, we obtain that the process is differentiable in the mean square sense. As a consequence we could apply standard integration techniques for such cases. 
\item
Finally, the fourth assumption is quite mild as for it we simply need that when $s$ and $t$  are both close to each 
other and at the same time near to zero, 
the variance $V(s)$ is not ''too far'' from the covariance $R(s,t)$. The fourth condition is also connected to the notion
of \emph{local non-determinism} introduced by Berman \cite{b} in connection to the existence of local time as an occupation density.
\end{enumerate}
\end{rmk}

\begin{exm}
Let $0\le s\le t\le T$.
\begin{enumerate}
\item
For processes with stationary increments 
\begin{eqnarray*}
R(t,s) &=& \frac{1}{2}\big[V(t)+V(s)-V(t-s)\big], \\
W(t,s) &=& V(t-s), \\
w^*(t) &=& V(t).
\end{eqnarray*}
Consequently, a process with stationary increments belongs to $\mathcal{X}^\alpha$ if and only 
if $V(t)>0$ for all $t>0$ and $V(t) = Ct^{2\alpha} + o(t^{2\alpha})$ at $t=0$.

In particular, the fractional Brownian motion with index $\alpha\in(0,1)$ belongs to the space $\mathcal{X}^\alpha$. 
\item
For stationary processes 
\begin{eqnarray*}
R(t,s) &=& r(t-s), \\
W(t,s) &=& 2\big[r(0)-r(t-s)\big], \\
V(t)   &=& r(0), \\
w^*(t) &=& 2\big[r(0)-r(t)\big].
\end{eqnarray*}
Consequently, a stationary process belongs to the class $\mathcal{X}^\alpha$ if and only if $r(t)>0$ for all $t$ and at $t=0$ we have
$$
r(0)-r(t) = Ct^{2\alpha} + o(t^{2\alpha}).
$$

In particular, fractional Ornstein--Uhlenbeck processes with index $\alpha\in(0,1)$ belongs to the space $\mathcal{X}^\alpha$. 
\end{enumerate}
\end{exm}


\subsection*{Existence of Generalized Lebesgue--Stieltjes Integral}
We begin with one of our main theorems which guarantees the existence of integrals.
\begin{thm}
\label{main_thm:stat}
Let $X\in\mathcal{X}^{\alpha}$ for some $\alpha>\frac{1}{3}$. If $\beta < \alpha \wedge (3\alpha -1)$, then for every convex function $f$ we have $f'_-(X_\cdot)\in W_2^{\beta}$ almost surely. 
\end{thm}

\begin{proof}
Let us first note that it is sufficient to consider convex functions of the form $f(x)=(x-a)^+$, $a\in\R$. Indeed, assume for a moment that the result 
is valid for these particular convex functions and  let $f$ be a general convex function for which the measure $f''$ has compact support. Then
\begin{equation}
\label{rep_compact}
|f'_-(X_t) - f'_-(X_s)| \leq \int_{\mathrm{supp}(f'')} \mathbf{1}_{X_t < a < X_s} + \mathbf{1}_{X_s < a < X_t}\, f''(\ud a).
\end{equation}
Now, by applying Tonelli's theorem, and the result for the functions $(x-a)^+$, $a\in\R$, 
we obtain the result for the convex functions $f$ with $\mathrm{supp}(f'')$ compact. Finally, the general case follows by approximating 
convex function $f$ with a sequence $f_n$ of convex functions for which $f''_n$ has compact support (see \cite{a-m-v} or \cite{a-v} for details). 

Consider then functions $f(x)=(x-a)^+$, $a\in\R$. Now,
\begin{equation*}
|f'_-(X_t) - f'_-(X_s)| = \textbf{1}_{X_t < a < X_s}+ \textbf{1}_{X_s < a < X_t}.
\end{equation*}
For the first term in the fractional Besov norm we obtain
\begin{equation*}
\int_0^T\frac{\textbf{1}_{X_t > a}}{t^{\beta}}\,\ud t \leq \int_0^T\frac{1}{t^{\beta}}\,\ud t < \infty.
\end{equation*}
It follows that we only have to prove that
\begin{equation*}
I=\int_0^T\!\int_0^t \frac{\mathbf{1}_{X_s<a<X_t}+\mathbf{1}_{X_t<a<X_s}}{(t-s)^{\beta+1}}\,\ud s\ud t< \infty.
\end{equation*}
Now, the iterated integral above can be split as
\begin{eqnarray*}
I &=& \int_0^{2\delta}\!\int_{\frac{t}{2}}^t\frac{\mathbf{1}_{X_s<a<X_t}+\mathbf{1}_{X_t<a<X_s}}{(t-s)^{\beta+1}}\,\ud s\ud t \\
& &+ \int_{2\delta}^T\!\int_{t-\delta}^t\frac{\mathbf{1}_{X_s<a<X_t}+\mathbf{1}_{X_t<a<X_s}}{(t-s)^{\beta+1}}\,\ud s\ud t\\
& &+ \int_{2\delta}^T\!\int_0^{t-\delta}\frac{\mathbf{1}_{X_s<a<X_t}+\mathbf{1}_{X_t<a<X_s}}{(t-s)^{\beta+1}}\,\ud s\ud t \\
& &+ \int_0^{2\delta}\!\int_0^{\frac{t}{2}}\frac{\mathbf{1}_{X_s<a<X_t}+\mathbf{1}_{X_t<a<X_s}}{(t-s)^{\beta+1}}\,\ud s\ud t\\
&=& I_1 + I_2 + I_3 + I_4.
\end{eqnarray*}
In integrals $I_3$ and $I_4$ we can bound indicators with one and they are still finite. Hence, it is sufficient to consider integrals $I_1$ and $I_2$. First, we note that by taking expectation and applying Tonelli's theorem we have the result if
\begin{equation}
\label{assumption_stupid2}
\int_{2\delta}^T\int_{t-\delta}^t \frac{\P(X_t<a<X_s)+\P(X_t>a>X_s)}{(t-s)^{\beta+1}}\,\ud s\ud t < \infty,
\end{equation}
and
\begin{equation}
\label{assumption_stupid}
\int_0^{2\delta}\int_{\frac{t}{2}}^t \frac{\P(X_t<a<X_s)+\P(X_t>a>X_s)}{(t-s)^{\beta+1}}\,\ud s\ud t < \infty.
\end{equation}
We begin with the term (\ref{assumption_stupid2}). By symmetry we can only analyze the term $\P(X_t>a>X_s)$. We have
\begin{equation*}
W(t,s)\leq w^*(t-s).
\end{equation*}
Now, since $X$ is continuous on $[0,T]$, $V^*<\infty$.
Hence, by assumption (iii) of Definition \ref{defn:Xalpha} and Lemma \ref{lma:crossing2}, it is sufficient to consider integral of form
\begin{equation*}
\int_{2\delta}^T\!\int_{t-\delta}^t \frac{\sqrt{W(t,s)}}{(t-s)^{\beta+1}}\,\ud s\ud t.
\end{equation*}
Let now $\delta$ be small enough such that, by assumption (ii) of Definition \ref{defn:Xalpha}, we have
\begin{equation*}
w^*(t) \leq Ct^{2\alpha}
\end{equation*}
for some constant $C$. Since $\alpha>\beta$ we obtain that (\ref{assumption_stupid2}) holds almost surely. 
Consider next term in (\ref{assumption_stupid}). Note first that by applying Lemma \ref{lma:crossing2} together with (iii) and (iv) of Definition \ref{defn:Xalpha} we obtain
\begin{equation*}
\int_0^{2\delta}\!\int_{\frac{t}{2}}^t \frac{(t-s)^{\alpha}}{s(t-s)^{\beta+1}}\,\ud s \ud t 
\,\leq\, C\int_0^{2\delta} t^{-1}t^{\alpha-\beta}\,\ud t \,<\, \infty.
\end{equation*}
Hence it is sufficient to study a term of form
\begin{equation*}
e^{-\frac{a^2}{2V^*}}\frac{|a|\sqrt{W(t,s)}}{V(s)}.
\end{equation*}
The case $a=0$ is trivial. Let now $a\neq 0$ and introduce time points $t_0=2\delta$, 
$$
t_k = 2\delta- 2\delta\sum_{j=1}^k \left(\frac{1}{2}\right)^j, \quad k\geq 1.
$$
Now the integral can be split as
\begin{equation*}
\int_0^{2\delta}\int_{\frac{t}{2}}^t \ldots\ud s\ud t =
\sum_{k=0}^\infty \int_{t_{k+1}}^{t_k}\int_{\frac{t}{2}}^t\ldots \ud s\ud t.
\end{equation*}
Note next that assumptions (i) and (iii) of Definition \ref{defn:Xalpha} implies that either $\inf_{s\in[0,T]}V(s)>0$ or else $X_0$ is a constant. In the first case the 
proof is trivial. On the latter case, since $X$ is centered, we have $X_0=0$ and in this case 
$$
\sup_{t_{k+1}\leq s\leq t_k}V(s) \leq Ct_k^{2\alpha}.
$$
Hence by applying Lemma \ref{lma:crossing2} we obtain
\begin{eqnarray*}
\sum_{k=0}^\infty \int_{t_{k+1}}^{t_k}\int_{\frac{t}{2}}^t\frac{\P(X_t>a>X_s)}{(t-s)^{\beta+1}} \ud s\ud t 
&\leq& C |a|\sum_{k=0}^\infty \int_{t_{k+1}}^{t_k}\int_{\frac{t}{2}}^t\frac{e^{-\frac{a^2}{2Ct_{k}^{2\alpha}}}(t-s)^{\alpha}}{s^{2}(t-s)^{\beta+1}} \ud s\ud t \\
&\leq& C \sum_{k=0}^\infty  \int_{t_{k+1}}^{t_k} e^{-\frac{a^2}{2Ct_{k}^{2\alpha}}}t^{-2}t^{\alpha-\beta}\ud t\\
&\leq& C \sum_{k=0}^\infty e^{-\frac{a^2}{2Ct_{k}^{2\alpha}}}t_k^{\alpha-\beta-1}\\
&<& \infty.
\end{eqnarray*}
Hence the result is valid for convex function $f(x)=(x-a)^+$. 
Consider next more general convex function for which $f''$ has compact support. Using (\ref{rep_compact}) we have
\begin{eqnarray*}
& &\int_{\mathrm{supp}(f'')} \mathbf{1}_{X_t < a < X_s} + \mathbf{1}_{X_s < a < X_t}\, f''(\ud a)\\
&= & \int_{\mathrm{supp}(f''): a\in[0,1]} \mathbf{1}_{X_t < a < X_s} + \mathbf{1}_{X_s < a < X_t}\, f''(\ud a)\\
&+&\int_{\mathrm{supp}(f''): a\in[-1,0)} \mathbf{1}_{X_t < a < X_s} + \mathbf{1}_{X_s < a < X_t}\, f''(\ud a)\\
&+& \int_{\mathrm{supp}(f'') : |a|>1} \mathbf{1}_{X_t < a < X_s} + \mathbf{1}_{X_s < a < X_t}\, f''(\ud a)\\
&=:& J_1 + J_2 + J_3.
\end{eqnarray*}
For $J_3$ we have lower bound for $a$ and since $\int_{\mathrm{supp}(f'') : |a|>1}f''(\ud a) < C$ the result follows immediately from the result for functions 
$(x-a)^+$. Consider next the term $J_1$ and take a sequence of functions $f_n\in C^2$ such that 
\begin{equation*}
\int_\R g(x)f''_n(x)\ud x \rightarrow \int_\R g(x)f''(\ud x)
\end{equation*}
for every function $g\in C^1$ with compact support (see \cite{a-m-v}). Then noting that $3\alpha > \beta +1$ we get
\begin{eqnarray*}
& &\sum_{k=0}^\infty \int_0^1ae^{-\frac{a^2}{2Ct_{k}^{2\alpha}}}t_k^{\alpha-\beta-1}f''(\ud a) \\
& = & \sum_{k=0}^\infty \lim\sup_n\int_0^1ae^{-\frac{a^2}{2Ct_{k}^{2\alpha}}}t_k^{\alpha-\beta-1}f''_n(a)\ud a \\
&\leq & \sum_{k=0}^\infty\lim\sup_n\max_{0\leq a \leq 1}f''_n(a) \int_0^1ae^{-\frac{a^2}{2Ct_{k}^{2\alpha}}}t_k^{\alpha-\beta-1}\ud a \\
&\leq & C\lim\sup_n\max_{0\leq a \leq 1}f''_n(a)\sum_{k=0}^\infty \int_0^1 a e^{-\frac{a^2}{2Ct_{k}^{2\alpha}}}t_k^{\alpha-\beta-1}\ud a \\
&\leq & C\lim\sup_n\max_{0\leq a \leq 1}f''_n(a)\sum_{k=0}^\infty t_k^{3\alpha-\beta-1}\\
&\leq & C\lim\sup_n\max_{0\leq a \leq 1}f''_n(a)\\
& < & \infty
\end{eqnarray*}
since for every fixed $a$ we have $f''_n(a)\rightarrow f''(a)$ and we consider maximum on a compact interval (see \cite{a-m-v} for details). 
The term $J_2$ can be treated similarly and hence we are done.
\end{proof}
Now it is obvious to obtain the existence of the integral by Proposition \ref{pr:n-r}.
\begin{cor}
Let $f$ be a linear combination of convex functions, $Y\in W_1^{1-\beta}$ and $X\in\mathcal{X}^{\alpha}$ for some $\alpha>\frac{1}{3}$. If $\beta < \alpha \wedge (3\alpha -1)$, then the integral
\begin{equation*}
\int_0^T f'_-(X_u)\,\ud Y_u
\end{equation*}
exists almost surely in the sense of generalized Lebesgue-Stieltjes integral.
\end{cor}

\begin{cor}
Let $f$ be a linear combination of convex functions, $Y\in W_1^{1-\beta}$ and $X\in\mathcal{X}^{\alpha}$ for some $\alpha>\frac{1}{3}$. Put $S=g(X)$ for some strictly monotone function $g\in C^2$. If $\beta<\alpha \wedge (3\alpha -1)$, then the integral
\begin{equation*}
\int_0^T f'_-(S_u)g'(X_u)\,\ud Y_u
\end{equation*}
exists almost surely in the sense of generalized Lebesgue-Stieltjes integral.
\end{cor}

\begin{proof}
Without loss of generality we can assume that $g$ is strictly increasing. Now the inverse $g^{-1}$ exists and we have
\begin{equation*}
\mathbf{1}_{S_t > a > S_s} = \mathbf{1}_{X_t > g^{-1}(a) > X_s}.
\end{equation*}
Hence, by following the proof in \cite{a-m-v} we obtain the result from our main theorem if
\begin{equation*}
\int_0^T \int_0^t \frac{|f'_-(S_t)||g'(X_t) - g'(X_s)|}{(t-s)^{\beta+1}}\ud s \,\ud t < \infty.
\end{equation*}
We obtain by Taylor's theorem that 
\begin{equation*}
|g'(X_t) - g'(X_s)| = |g''(\xi)||X_t-X_s|
\end{equation*}
for some random point $\xi$ between $X_s$ and $X_t$. It remains to note that
\begin{equation*}
\E|X_t-X_s| \leq \sqrt{\E|X_t-X_s|^2} =\sqrt{W(t,s)},
\end{equation*}
and the claim follows.
\end{proof}

\begin{rmk}
For financial applications natural candidate is $g(x)=e^x$.
\end{rmk}

Next our aim is to define integrals over the random interval $[0,\tau]$, where $\tau$ is an almost surely finite random variable instead of deterministic time. 

\begin{thm}
Let $X_t\in\mathcal{X}^{\alpha}$ for some $\alpha>\frac{1}{3}$ and let $\tau\leq T$ be a random time. If $\beta<\alpha \wedge(3\alpha-1)$, then for every convex function $f$ we have $f'_-(X_\cdot)\mathbf{1}_{\cdot\leq \tau}\in W_2^{\beta}$ almost surely. 
\end{thm}

\begin{proof}
For the first term in the fractional Besov norm we can simply make upper bound
\begin{equation*}
\mathbf{1}_{t\leq \tau} \leq 1
\end{equation*}
and we can proceed as in the proof of our main theorem \ref{main_thm:stat}. 

Consider then the second term in the fractional Besov norm:
\begin{equation*}
\int_0^T\int_0^t \frac{|f'_-(X_t)\textbf{1}_{t\leq \tau}-f'_-(X_s)\textbf{1}_{s\leq \tau}|}{(t-s)^{\beta+1}}\ud s\ud t.
\end{equation*}
Now either $s\leq t\leq\tau$ in which case we can proceed as for deterministic time or we have $t>\tau$ and $s<\tau$. In this case we get
\begin{eqnarray*}
\int_{\tau}^T\!\int_0^{\tau}\frac{|f'_-(X_s)|}{(t-s)^{\beta+1}}\,\ud s\ud t
&\leq& C(\beta)\sup_{s\in[0,T]}|f'_-(X_s)| \int_{\tau}^T (u-\tau)^{-\beta}\,\ud u\\
&\leq& C(\beta)\sup_{s\in[0,T]}|f'_-(X_s)| T^{1-\beta}.
\end{eqnarray*}
This completes the proof.
\end{proof}
Consequently, the existence of the integral is now evident.
\begin{cor}
Let $f$ be a linear combination of convex functions, $Y\in W_1^{1-\beta}$, $X\in\mathcal{X}^{\alpha}$ for some $\alpha>\frac{1}{3}$, and let $\tau\leq T$ be a bounded random time. If $\beta<\alpha \wedge(3\alpha-1)$, then the integral
\begin{equation*}
\int_0^\tau f'_-(X_u)\,\ud Y_u
\end{equation*}
exists almost surely in the sense of generalized Lebesgue--Stieltjes integral.
\end{cor}

\begin{rmk}
Let $g$ be strictly monotone, and set $S=g(X)$. Then similar results hold for integrals
\begin{equation*}
\int_0^\tau f'_-(S_u)g'(X_u)\,\ud Y_u.
\end{equation*}
\end{rmk}

\subsection*{Existence of Föllmer Integral}

As noted e.g. by Zähle \cite{z} we can sometimes approximate the generalized Lebesgue--Stieltjes integral with Riemann--Stieltjes sums. This is the topic of Theorem \ref{thm:rs} below. The proof follows exactly the same arguments as for the particular case for fractional Brownian motion in \cite{a-m-v}. Hence we only give the idea of the proof and details are left to the reader.

\begin{thm}\label{thm:rs}
Let $f$ be a linear combination of convex functions, $Y\in W_1^{1-\beta}$ and $X\in\mathcal{X}^{\alpha}$ for some $\alpha>\frac{1}{3}$. If $\beta<\alpha\wedge(3\alpha-1)$, then 
\begin{equation*}
\sum_{j=1}^{k(n)} f'_-(X_{u_j^n})(Y_{u_{j+1}^n}-Y_{u_j^n}) \rightarrow \int_0^T f'_-(X_u)\ud Y_u
\end{equation*}
almost surely for any partition $\pi_n = \{0=u_0^n<\ldots<u_{k(n)}^n=T\}$ such that $|\pi_n|\rightarrow 0$.
\end{thm}

\begin{proof}
Again we can assume that the measure $f''$ has compact support. Now,
\begin{equation*}
\int_0^T f'_-(X_u)\,\ud Y_u - \sum_{j=1}^{k(n)} f'_-(X_{u_j^n})(Y_{u_{j+1}^n}-Y_{u_j^n}) 
= \int_0^T h_n(X_u) \ud Y_u,
\end{equation*}
where
\begin{equation*}
h_n(u)=\sum_{j=1}^{k(n)}\left(f'_-(X_{u_j^n}) - f'_-(X_u)\right)\mathbf{1}_{u_j^n<u\leq u_{j+1}^n}.
\end{equation*}
To conclude the proof we have to show that
$$
{\|h_n\|}_{\beta,2} \to 0
$$
almost surely. Following arguments in \cite{a-m-v} we obtain an integrable upper bound in both terms of the fractional Besov norm, and hence the result follows by using the dominated convergence theorem. More precisely, we have
\begin{equation*}
|h_n(t)| \leq 2 \sup_{0\leq u\leq T}|f'_-(X_u)|,
\end{equation*}
and
\begin{equation*}
|h_n(t) - h_n(s)| \leq C\int \textbf{1}_{X_s<a<X_t}\, f''(\ud a)
\end{equation*}
on set $\{X_s < a < X_t\}$ and
\begin{equation*}
|h_n(t) - h_n(s)| \leq C\int \textbf{1}_{X_t<a<X_s}\,f''(\ud a)
\end{equation*}
on set $\{X_t < a < X_s\}$. Hence we have integrable upper bound in all the cases and hence the result follows by dominated convergence theorem together with the fact that
\begin{equation*}
|h_n(t)| \to 0
\end{equation*}
almost surely.
\end{proof}

\begin{rmk}
Let $g$ be strictly monotone and set $S=g(X)$. Let $\tau\le T$ be a random time.  Then similar result holds for integrals of the form
\begin{equation*}
\int_0^\tau f'_-(S_u)g'(X_u)\, \ud Y_u.
\end{equation*}
\end{rmk}

\subsection*{Existence of Mixed Integrals}

The particular reason why we considered integrals of form
$
\int_0^T f'_-(X_u)\ud Y_u
$ with arbitrary process $Y$  instead of $X$ is that now we can apply our result for processes of type
$$
Y = M + X
$$ 
where $M$ is a centered Gaussian martingale with a Lipschitz continuous variance function $\la M\ra$. These kind of mixed processes are especially interesting in mathematical finance, see \cite{b-s-v-0,b-s-v}.

Note that Lipschitz continuity of the variance on $M$ implies that
$$
\E[w_M^*(t)]\leq Ct.
$$

\begin{thm}
\label{thm:existence_summa}
Let $M$ be a centered Gaussian martingale with a Lipschitz variance and let $X$ be an $\alpha$-H\"older continuous Gaussian process for some $\alpha>1/2$. Denote $Y=M+X$. Let $f$ be a linear combination of convex functions and let $\tau\le T$ be a stopping time. Assume $Y$ satisfies assumptions (i), (iii) and (iv) of Definition \ref{defn:Xalpha}. Furthermore, assume $Y$ is adapted to the filtration $\mathcal{F}^M$ generated by $M$. Then the integral 
\begin{equation}\label{bilin}
\int_0^\tau f_-'(Y_u)\,\ud Y_u = \int_0^\tau f_-'(Y_u)\,\ud M_u + \int_0^\tau f_-'(Y_u)\,\ud X_u
\end{equation}
exists as a Föllmer integral.
\end{thm}

\begin{rmk}

Note that we do not assume that the processes $M$ and $X$ are independent. 

\end{rmk}

\begin{proof}
Consider first the first integral on the right-hand side of (\ref{bilin}).  Since $M$ is a martingale this integral exists as an It\^o integral, and by using suitable subdivision sequence $(\pi_n)_{n=0}^\infty$ it exists pathwise as a Föllmer integral.

Consider then the second integral on the right-hand side of (\ref{bilin}).  Since $M$ has Lipschitz variance we have that $Y\in\mathcal{X}^{1/2}$. 
Consequently, $f_-'(Y)\mathbf{1}_{\cdot\le \tau}\in W^\beta_2$ for all $\beta\in(1-\alpha,1/2)$.  Therefore, the generalized Lebesgue--Stieltjes integral 
$
\int_0^\tau f'_-(Y_u)\,\ud X_u
$
exists.

Finally, all the integrals in (\ref{bilin}) can also be understood as Föllmer integrals due to Theorem \ref{thm:rs}.
\end{proof}

\subsection*{It\^o--Tanaka Formula}

We begin with the following It\^o formula for smooth functions. The proof is based on Taylor expansion and is the same as in the semimartingale case, or indeed, in the classical case.

\begin{prop}\label{prop:ito}
Let $X\in\mathcal{X}^{\alpha}$ with $\alpha>\frac{1}{2}$ and let $f\in C^2$. Then 
$$
f(X_T)=f(X_0) + \int_0^T f'(X_u)\,\ud X_u 
$$ 
\end{prop}

The It\^o formula of Proposition \ref{prop:ito} can be extended to convex functions.  Indeed, the arguments presented in \cite{a-m-v} imply that:

\begin{thm}
\label{thm:ito}
Let $X\in\mathcal{X}^{\alpha}$ with $\alpha>\frac{1}{2}$ and let $f$ be a linear combination of convex 
functions. Then 
$$
f(X_T) = f(X_0) + \int_0^T f'_-(X_u)\,\ud X_u.
$$
\end{thm}

\begin{cor} 
Let $g$ be a strictly monotone smooth function and set $S=g(X)$ for some $X\in\mathcal{X}^\alpha$ with $\alpha>\frac{1}{2}$. Let $f$ be a linear combination of convex functions and let $\tau\le T$ be a random time.  Then 
$$
f(S_\tau) = f(S_0) + \int_0^\tau f'_-(S_u)\,\ud S_u.
$$
\end{cor} 

Let us now consider the non-smooth case. 

Föllmer \cite{Follmer} showed that for any process $Y$ with finite quadratic variation $\la Y\ra$ and $f\in C^2$  we have
$$
f(Y_T)=f(Y_0) + \int_0^T f'(Y_u)\,\ud Y_u + \frac{1}{2}\int_0^T f''(Y_u) \,\ud \la Y\ra_u,
$$
which also implies the existence of the Föllmer integral $\int_0^T f'(Y_u)\,\ud Y_u$. We will extend this result to convex functions $f$.  We will, however, not consider general quadratic variation processes. Instead, we consider processes that are of the form $Y = M + X$ as in Theorem \ref{thm:existence_summa}.

The non-trivial quadratic variation gives rise to local time:

\begin{defn}
Let $Y$ be a continuous process with quadratic variation $\la Y\ra$. Its \emph{local time} $L^a_t(Y)$ is the process defined by the \emph{occupation time formula}
\begin{equation}\label{otf}
\int_{-\infty}^\infty g(a)L_t^a(Y)\,\ud a = \int_0^t g(Y_u)\,\ud \la Y \ra_u
\end{equation}
almost surely for every bounded, Borel measurable function $g$.
\end{defn}

\begin{thm}\label{thm:mg_ito}
Let $f$ and $Y=M+X$ satisfy conditions of Theorem \ref{thm:existence_summa}. Then there exists a local time process $L_t^a(Y)$ such that 
\begin{equation}
\label{ito:sum}
f(Y_t) = f(Y_0) + \int_0^t f'_-(Y_u)\,\ud Y_u + \frac{1}{2}\int_{-\infty}^{\infty} L_t^a(Y)\,f''(\ud a).
\end{equation}
\end{thm}

\begin{proof}
Note first that Theorem \ref{thm:existence_summa} implies the existence of integrals as mixed Föllmer and generalized Lebesgue--Stieltjes integral.  By Taylor expansion we have the result for $f\in C^2$. For the general case we follow arguments for the semimartingale case(see \cite{r-y}, pp. 221--224).

Let now $f$ be a convex function for which the measure $f''$ has compact support and define
\begin{equation*}
f_n(x)= n\int_{-\infty}^0 f(x+y)j(ny)\,\ud y,
\end{equation*}
where $j(y)$ is a positive $C^\infty$-function with compact support in $(-\infty,0]$ such that 
\begin{equation*}
\int_{-\infty}^0 j(y)\,\ud y = 1.
\end{equation*}
Now $f_n$ is well-defined and smooth for every $n$. Moreover, $f_n$ converges to $f$ pointwise and $(f_n)'_-$ increases to $f'_-$. Now, 
\begin{equation*}
f_n(Y_t) = f_n(Y_0) + \int_0^t (f_n)'_-(Y_u)\,\ud Y_u + \frac{1}{2}\int_0^t f_n''(Y_u)\,\ud\la M \ra_u.
\end{equation*}
Moreover, by examining the proof of Theorem \ref{thm:ito} we can conclude that
\begin{equation*}
\int_0^t (f_n)'_-(Y_u)\ud X_u \rightarrow \int_0^t f'_-(Y_u)\,\ud X_u,
\end{equation*}
and for the It\^o integral we get by standard arguments that
\begin{equation*}
\int_0^t (f_n)'_-(Y_u)\,\ud M_u \rightarrow \int_0^t f'_-(Y_u)\,\ud M_u.
\end{equation*}
Consequently, we obtain that for every convex function $f$ there exists a process $A_f$ such that
\begin{equation*}
f(Y_t) = f(Y_0) + \int_0^t f'_-(Y_u) \,\ud Y_u + \frac{1}{2}A_f^t.
\end{equation*}
Applying this result to convex functions $(x-a)^+$ and $(x-a)^-$ we obtain
\begin{equation*}
(Y_t-a)^+ = (Y_0-a)^+ + \int_0^t \textbf{1}_{Y_u>a}\,\ud Y_u + \frac{1}{2}A_+^t(a)
\end{equation*}
and
\begin{equation*}
(Y_t-a)^- = (Y_0-a)^- - \int_0^t\textbf{1}_{Y_u\leq a}\,\ud Y_u + \frac{1}{2}A_-^t(a)
\end{equation*}
for some processes $A_+^t(a)$ and $A_-^t(a)$. Subtract the second equation from the first to get that
\begin{equation*}
A_+^t(a) = A_-^t(a)
\end{equation*}
almost surely, and we define the local time process $L_t^a(Y) = A_+^t(a)$. Moreover, adding the second equation to the first we get
\begin{equation*}
|Y_t - a| = |Y_0 - a| + \int_0^t \mathrm{sgn}(Y_u-a)\,\ud Y_u + L_t^a(Y).
\end{equation*}
In order to have (\ref{ito:sum}) for convex function $f$ for which $f''$ has compact support we use representations
$$
f(x) = \alpha x + \beta + \int_{-\infty}^\infty |x-a|\,f''(\ud a)
$$
and 
\begin{equation*}
f'_-(x) = \alpha + \int_{-\infty}^\infty \mathrm{sgn}(x-a)\,f''(\ud a).
\end{equation*}
By using the representation formula for the convex function $|Y_t -a|$ we obtain
\begin{eqnarray*}
f(Y_t) &=& f(Y_0) + \alpha(Y_t - Y_0) \\
& &+ \int_{-\infty}^\infty\frac{1}{2}\left(\int_0^t \mathrm{sgn}(Y_u-a)\,\ud Y_u + L_t^a(Y)\right)\,f''(\ud a).
\end{eqnarray*}
Hence, for convex functions $f$ for which $f''$ has compact support, the process $A_f$ is given by
\begin{eqnarray*}
A_f^t &=& \int_{-\infty}^\infty L_t^a(Y)\,f''(\ud a) + \int_{-\infty}^\infty\!\int_0^t \mathrm{sgn}(Y_u - a)\,\ud Y_u\,f''(\ud a)\\
& & - \int_0^t f'_-(Y_u)\,\ud Y_u + \alpha(Y_t - Y_0).
\end{eqnarray*}
It remains to apply Fubini's Theorem to stochastic integrals. For the martingale part we use classical Stochastic Fubini's Theorem (see \cite{r-y}) and for pathwise integral we can use classical Fubini's Theorem for $\sigma$-finite measures. Now, we obtain the occupation time formula by the same argument as in the semimartingale case.
\end{proof}

\begin{cor}
\label{cor:occupation_density}
Let $M$ and $X$ be processes such that assumptions of Theorem \ref{thm:existence_summa} are satisfied.  Then the local time $L_t^a(Y)$ of the process $Y = X + M$ is almost surely continuous in $t$ and in $a$ and we have the \emph{local representation}
$$
L_t^a(Y) = \lim_{\epsilon\to0+}\frac{1}{2\epsilon}
\int_0^t \mathbf{1}_{a-\epsilon < Y_u < a+\epsilon}\, \ud\la Y\ra_u
$$
almost surely.
\end{cor}

\begin{proof}
Obviously $(Y_t - a)^+$ is continuous in $t$ and $a$. Hence it is enough to show that the stochastic integral 
\begin{equation*}
\int_0^t \textbf{1}_{Y_u > a}\,\ud Y_u
\end{equation*}
is continuous in $t$ and $a$. For this it is enough to show that
\begin{equation*}
{\|\mathbf{1}_{X_u>a}\mathbf{1}_{u\leq t}\|}_{2,\beta}
\end{equation*}
is continuous in $t$ and $a$. But continuity in $t$ is evident and the continuity in $a$ follows from Lebesgue dominated convergence theorem. 

Finally, the local representation follows from the continuity of the local time in $a$ and the occupation time formula (\ref{otf}). 
\end{proof}

\begin{rmk}
\begin{enumerate}
\item
In the case of a standard Brownian motion $W$ with $g(x)=\mathbf{1}_A$ the occupation time formula reads
that 
$$
\int_A L_t^a(W)\,\ud a = \int_0^t \mathbf{1}_{\{W_u\in A\}}\,\ud u.
$$
So the local time is exactly the density with respect to Lebesgue measure for the occupation measure.  In the It\^o--Tanaka formula the local time for $Y$ is the density of the occupation measure with respect to a ``clock'' $\ud\la Y\ra$.  
\item
Berman \cite{b} showed that a Gaussian process $X$ admits a local time as occupation density with respect to 
the ``Lebesgue clock'' $\ud u$ if and only if its incremental variance $W$ satisfies
$$
\int_0^T\!\int_0^T \frac{1}{\sqrt{W(t,s)}}\,\ud s\ud t <\infty.
$$
Consequently, every stationary or stationary increment process $X\in\mathcal{X}^{\alpha}$ admits a local time defined as the density of the occupation 
measure with respect to the ``Lebesgue clock''. If $\alpha>1/2$, then $L_T^a(X) =0$.
\item
In Corollary \ref{cor:occupation_density} we have 
$$
\ud\la Y\ra_u = \ud\la M\ra_u = \left(\la M\ra'_-\right)_u\ud u.
$$
\item
Again with similar arguments we can obtain the It\^o--Tanaka formula for transformation $S=g(Y)$ with obvious changes.
\end{enumerate}
\end{rmk}

\section{Implications to Option-Pricing}
\label{sect:finance}

We explain implications of the It\^o--Tanaka formula for pricing and hedging of European options in the spirit of Sondermann \cite[Ch. 6]{Sondermann}.  For the use of non-semimartingales and pathwise Föllmer integration in mathematical finance we refer to \cite{b-s-v}.

Let $S$ be the discounted stock-price process given by dynamics
$$
\frac{\ud S_t}{S_t} = \mu(t)\ud t + \ud Y_t, 
$$
where $Y$ is a centered continuous quadratic variation process. Then, by \cite{b-s-v},
$$
S_t = S_0\exp\left\{\int_0^t \mu(u)\,\ud u + Y_t - \frac12\la Y\ra_t\right\}.
$$

Suppose we want to replicate a European call-option $(S_T-K)^+$ and suppose $Y\in\mathcal{X}^\alpha$ for some $\alpha>1/2$. Then $\la Y\ra=0$, and the It\^o--Tanaka formula takes the form
$$
(S_T-K)^+ = (S_0-K)^+ + \int_0^T \mathbf{1}_{S_u\ge K}\, \ud S_u.
$$
This means that one can replicate the European call-option $(S_T-K)^+$ with a fair price $(S_0-K)^+$ by dynamically buying at time $t$ one share of the stock if the option goes from out-of-the-money into in-the-money and selling the stock if the option goes from out-of-the-money into in-the-money. But this is silly because of at least two reasons:
\begin{enumerate}
\item
Take $K>S_0$. Then out-of-the-money options are worthless. This is obviously nonsense!
\item
Take $K=S_0$.  Then one could make profit without risk by selling cheap and buying expensive.  This is against any reasonable business sense! 
\end{enumerate}
Thus, we see that in order to avoid this \emph{buy--sell paradox} the option-pricing model must include a quadratic variation. 

Suppose then that $Y=M+X$ such that the assumption of Theorem \ref{thm:mg_ito} are satisfied.  Then the It\^o--Tanaka formula takes the form
\begin{equation}\label{eq:hedge}
(S_T-K)^+ = (S_0-K)^+ + \int_0^T \mathbf{1}_{S_u\ge K}\, \ud S_u 
+\frac12 L_T^K(S).
\end{equation}
This solves the buy--sell paradox.  Indeed, now the buy--sell strategy $\mathbf{1}_{S_u\ge K}$ is no longer a self-financing hedging strategy for the 
call-option $(S_T-K)^+$. 

\begin{rmk}
\begin{enumerate}
\item 
The local time $\frac12 L_T^K(S)$ in the hedging formula (\ref{eq:hedge}) can be interpreted as transaction costs in a following way noted by Sondermann \cite{Sondermann}: Assume that one tries to apply the buy--sell strategy $\mathbf{1}_{S_u\ge K}$ to hedge the European call-option $(S_T-K)^+$, i.e., buy the stock at price $K$ when up-crossing the barrier $K$, sell it again when down-crossing the barrier. But you cannot sell it at the same price. You need a so-called \emph{cutout}. You can place only limit orders of the form: buy at $K$, sell at $K-\epsilon$ for some $\epsilon>0$. The smaller you choose $\epsilon$, the more cutouts you will face, and in the limit the sum of these cutouts is just equal to transaction costs.  
\item
The true hedging strategy for the European call-option in the models $Y=M+X$ satisfying the assumptions of Theorem \ref{thm:mg_ito} can be calculated just like in the martingale case $Y=M$ by using the Black--Scholes BPDE.  See \cite{b-s-v} for details. 
\end{enumerate}
\end{rmk}

\appendix
\section{Level-Crossing Lemma}

The key lemma in our analysis is the following estimate for the probability that a Gaussian process $X$ crosses a fixed level.  Actually, in \cite{a-v} the authors proved the lemma in the particular case of fractional Brownian motion. We extend the result here for more general Gaussian process. We consider only probability $\P(X_s < a < X_t)$ and a case $\sup_{s\leq T} V(s) \leq 1$.  However, by considering processes $Y = -X$ and $Y = \frac{X}{\sqrt{V^*}}$ we obtain same bound for $\P(X_s > a > X_t)$ and for the general 
case $\sup_{s\leq T} V(s) < \infty$.  Also, note that continuous Gaussian processes on compact time intervals satisfy $V^*<\infty$. 

\begin{lma}
\label{lma:crossing}
Let $X$ be a Gaussian process with positive covariance function $R$. Denote
\begin{equation*}
\sigma^2 = \frac{R(s,s)R(t,t) - R(t,s)^2}{R(s,s)},
\end{equation*}
and fix $0 < s < t \le T$ and $ a \in \R$. Assume that the variance function satisfies
\begin{equation*}
V^* := \sup_{s_\leq T} V(s) \leq 1.
\end{equation*}
\begin{enumerate}
\item
If
\begin{equation*}
\frac{R(s,s)}{R(t,s)}(a-1) < a,
\end{equation*}
then there exists a constant $C$, independent of $s$, $t$, $T$ and $a$, such that 
\begin{equation*}
\P \big( X_s < a < X_t\big) \leq I_1 + I_2 + I_3 +I_4,
\end{equation*}
where
\begin{eqnarray*}
I_1 &\leq& C \min[\sqrt{V(s)}\sigma,\sigma^2] e^{-\frac{a^2}{2}}, \\
I_2 &\leq& Ce^{-\frac{\min[a^2,(a-1)^2]}{2V^*}}\frac{\sigma}{\sqrt{V(s)}}\left[\mathbf{1}_{|a|>2}+\left(a-\frac{R(s,s)}{R(t,s)}(a-1)\right)\mathbf{1}_{|a|\leq 2}\right], \\
I_3 &\leq& C \frac{R(s,s)}{R(t,s)}\frac{\sigma}{\sqrt{V(s)}}e^{-\frac{\min[a^2,(a-1)^2]}{2V^*}}, \\
I_4 &\leq& e^{-\frac{a^2}{2V^*}}\frac{1}{\sqrt{V(s)}}\left|a\left(1-\frac{R(s,s)}{R(t,s)}\right)\right|,
\end{eqnarray*}
\item
If 
\begin{equation*}
\frac{R(s,s)}{R(t,s)}(a-1) \geq a,
\end{equation*}
then there exists a constant $C$, independent of $s$, $t$, $T$ and $a$, such that 
\begin{equation*}
\P \big( X_s < a <X_t \big) \leq C \min[\sqrt{V(s)}\sigma,\sigma^2] e^{-\frac{a^2}{2}}.
\end{equation*}
\end{enumerate}
\end{lma}

In the proof we use the following well-known estimate.
\begin{lma}
\label{standard_estimate}
Let $Z$ be a standard normal random variable and fix $a>0$. Then
\begin{equation*}
\P\big(Z>a\big) \leq \frac{1}{\sqrt{2\pi}a}e^{-\frac{a^2}{2}}.
\end{equation*}
\end{lma}

\begin{proof}[Proof of Lemma \ref{lma:crossing}]
We make use of decomposition
\begin{equation*}
X_t = \frac{R(t,s)}{R(s,s)}X_s + \sigma Y,
\end{equation*}
where $Y$ is $N(0,1)$ random variable independent of $X_s$ and 
\begin{equation*}
\sigma^2 = \frac{R(t,t)R(s,s)-R(t,s)^2}{R(s,s)}.
\end{equation*}
Assume that 
$$
\frac{R(s,s)}{R(t,s)}(a-1) < a.
$$ 
Then we obtain 
\begin{eqnarray*}
\lefteqn{\P(X_s<a<X_t)}\\ 
&=& 
\int_{-\infty}^a \P\left(Y\geq \frac{a-\frac{R(t,s)}{R(s,s)}x}{\sigma}\right)\frac{1}{\sqrt{2\pi}\sqrt{V(s)}}e^{-\frac{x^2}{2V(s)}}\,\ud x\\
&=& \int_{-\infty}^{\frac{R(s,s)}{R(t,s)}(a-1)} \P\left(Y\geq \frac{a-\frac{R(t,s)}{R(s,s)}x}{\sigma}\right)\frac{1}{\sqrt{2\pi}\sqrt{V(s)}}e^{-\frac{x^2}{2V(s)}}\,\ud x\\
& &+ \int_{\frac{R(s,s)}{R(t,s)}(a-1)}^{a} \P\left(Y\geq \frac{a-\frac{R(t,s)}{R(s,s)}x}{\sigma}\right)\frac{1}{\sqrt{2\pi}\sqrt{V(s)}}e^{-\frac{x^2}{2V(s)}}\,\ud x\\
&=& I_1 + A_1.
\end{eqnarray*}
Moreover, if 
$$
\frac{R(s,s)}{R(t,s)}(a-1) \geq a,
$$ 
then
$$
\P(X_s<a<X_t) \leq I_1.
$$
Note that here $I_1$ corresponds the one given in the Lemma and $A_1$ contains $I_2$, $I_3$ and $I_4$. We shall use similar technique for the rest of the proof.

We begin with $I_1$.  By Lemma \ref{standard_estimate} we have
$$
\P\left(Y\geq \frac{a-\frac{R(t,s)}{R(s,s)}x}{\sigma}\right) \leq \frac{1}{\sqrt{2\pi}A(x)}e^{-\frac{A(x)^2}{2}},
$$
where $A(x) = \frac{a-\frac{R(t,s)}{R(s,s)}x}{\sigma}$. Hence
\begin{eqnarray*}
I_1 &\leq& 
\int_{-\infty}^{\frac{R(s,s)}{R(t,s)}(a-1)} \frac{1}{\sqrt{2\pi}A(x)}e^{-\frac{A(x)^2}{2}}\frac{1}{\sqrt{2\pi}\sqrt{V(s)}}e^{-\frac{x^2}{2V(s)}}\,\ud x\\
&\leq& \frac{\sigma}{\sqrt{V(s)}}e^{-\frac{a^2}{2}}\int_{-\infty}^{\frac{R(s,s)}{R(t,s)}(a-1)} \frac{1}{2\pi}e^{-\frac{A(x)^2}{2}-\frac{x^2}{2V(s)} + \frac{a^2}{2}}\,\ud x\\
&=& \frac{R(s,s)}{R(t,s)}\frac{\sigma}{\sqrt{V(s)}}e^{-\frac{a^2}{2}}\int_{1}^{\infty} \frac{1}{2\pi}e^{-\frac{y^2}{2\sigma^2}-\frac{\left[\frac{R(s,s)}{R(t,s)}(a-y)\right]^2}{2V(s)} + \frac{a^2}{2}}\,\ud y
\end{eqnarray*}
We proceed to study the integral. Now,
\begin{eqnarray*}
\lefteqn{-\frac{y^2}{2\sigma^2}-\frac{\left[\frac{R(s,s)}{R(t,s)}(a-y)\right]^2}{2V(s)} + \frac{a^2}{2} }\\
&=& -\frac{1}{2\bar{\sigma}^2}\left[\left(y-a\frac{R(s,s)}{R(t,s)^2}\bar{\sigma}^2\right)^2 + a^2\left(\frac{R(s,s)}{R(t,s)^2}\bar{\sigma}^2 - \bar{\sigma}^2 - \frac{R(s,s)^2}{R(t,s)^4}\bar{\sigma}^4\right)\right],
\end{eqnarray*}
where 
\begin{equation*}
\frac{1}{\bar{\sigma}^2}=\frac{1}{\sigma^2} + \frac{R(s,s)}{R(t,s)^2}.
\end{equation*}
Now 
\begin{equation*}
\frac{1}{\bar{\sigma}^2}\geq 1
\end{equation*}
and since $V^*\leq 1$, we also have
\begin{equation*}
\left(\frac{R(s,s)}{R(t,s)^2}\bar{\sigma}^2 - \bar{\sigma}^2 - \frac{R(s,s)^2}{R(t,s)^4}\bar{\sigma}^4\right) \geq 0.
\end{equation*}
Thus,
\begin{eqnarray*}
\lefteqn{\int_{1}^{\infty} \frac{1}{2\pi}e^{-\frac{y^2}{2\sigma^2}-\frac{\left[\frac{R(s,s)}{R(t,s)}(a-y)\right]^2}{2s^{2H}} + \frac{a^2}{2}}\,\ud y}\\
&\leq& \int_{1}^{\infty} \frac{1}{2\pi}e^{-\frac{1}{2\bar{\sigma}^2}\left(y-a\frac{R(s,s)}{R(t,s)^2}\bar{\sigma}^2\right)^2}\,\ud y\\
&\leq& \frac{\bar{\sigma}}{\sqrt{2\pi}}.
\end{eqnarray*}
Hence, we obtain for $I_1$ that
\begin{equation*}
I_1 \leq C\frac{R(s,s)}{R(t,s)}\frac{\sigma}{\sqrt{V(s)}}e^{-\frac{a^2}{2}}\bar{\sigma}. 
\end{equation*}
Now we have
\begin{equation*}
\bar{\sigma}^2 = \frac{\sigma^2R(t,s)^2}{\sigma^2+R(s,s)},
\end{equation*}
and hence for $I_1$ there exists a constant $C$ such that
\begin{equation*}
I_1 \leq C\min[\sqrt{V(s)}\sigma,\sigma^2]e^{-\frac{a^2}{2}}.
\end{equation*}
For the term $A_1$ we have
\begin{eqnarray*}
A_1 &=& \int_{\frac{R(s,s)}{R(t,s)}(a-1)}^{a} \int_{A(x)}^\infty \frac{1}{\sqrt{2\pi}}e^{-\frac{y^2}{2}}\,\ud y\frac{1}{\sqrt{2\pi}\sqrt{V(s)}}e^{-\frac{x^2}{2V(s)}}\,\ud x\\
& =& \int_{\frac{R(s,s)}{R(t,s)}(a-1)}^{a} \int_{A(x)}^{\frac{1}{\sigma}}\ldots \ud y\ud x + \int_{\frac{R(s,s)}{R(t,s)}(a-1)}^{a} \int_{\frac{1}{\sigma}}^\infty\ldots\ud y\ud x\\
&=& A_{2}+I_{2}. 
\end{eqnarray*}
Consider then $I_2$. Applying Lemma \ref{standard_estimate} we obtain
\begin{equation*}
I_2 \leq C \frac{\sigma}{\sqrt{V(s)}}e^{-\frac{1}{2\sigma^2}}\int_{\frac{R(s,s)}{R(t,s)}(a-1)}^{a} e^{-\frac{x^2}{2V(s)}}\,\ud x.
\end{equation*}
Note that $\sigma^2\geq 0$. Therefore,
\begin{equation*}
\frac{R(s,s)}{R(t,s)^2}\geq \frac{1}{R(t,t)}\geq \frac{1}{V^*}.
\end{equation*}
Now if $|a|> 2$, we can apply Lemma \ref{standard_estimate} to obtain
\begin{equation*}
\int_{\frac{R(s,s)}{R(t,s)}(a-1)}^{a} e^{-\frac{x^2}{2V(s)}}\,\ud x \leq Ce^{-\frac{\min[a^2,(a-1)^2]}{2V^*}}.
\end{equation*}
As a consequence, we obtain the required upper bound for $I_2$. Now, if $|a| \leq 2$ we obtain
\begin{equation*}
\int_{\frac{R(s,s)}{R(t,s)}(a-1)}^{a} e^{-\frac{x^2}{2V(s)}}\,\ud x \leq Ce^{-\frac{\min[a^2,(a-1)^2]}{2V^*}}\left[a-\frac{R(s,s)}{R(t,s)}(a-1)\right].
\end{equation*}
To conclude we study the term $A_2$. If we have
\begin{equation*}
\frac{R(s,s)}{R(t,s)}a < a,
\end{equation*}
then by applying the Tonelli theorem we obtain
\begin{eqnarray*}
A_2 &=& \int_{\frac{R(s,s)}{R(t,s)}(a-1)}^{a} \int_{A(x)}^{\frac{1}{\sigma}} \frac{1}{\sqrt{2\pi}}e^{-\frac{y^2}{2}}\,\ud y\frac{1}{\sqrt{2\pi}\sqrt{V(s)}}e^{-\frac{x^2}{2V(s)}}\,\ud x\\
&=& \int_{(\left[1-\frac{R(t,s)}{R(s,s)}\right]\frac{a}{\sigma}}^{\frac{1}{\sigma}}\int_{(a-\sigma y)\frac{R(s,s)}{R(t,s)}}^a \ldots \ud x\ud y\\
&=& \int_{(\left[1-\frac{R(t,s)}{R(s,s)}\right]\frac{a}{\sigma}}^{\frac{1}{\sigma}}\int_{(a-\sigma y)\frac{R(s,s)}{R(t,s)}}^{\frac{R(s,s)}{R(t,s)}a}\ldots \ud x \ud y + \int_{(\left[1-\frac{R(t,s)}{R(s,s)}\right]\frac{a}{\sigma}}^{\frac{1}{\sigma}}\int_{\frac{R(s,s)}{R(t,s)}a}^a\ldots\ud x\ud y\\
&=& I_3 + I_4.
\end{eqnarray*}
Moreover, if 
\begin{equation*}
\frac{R(s,s)}{R(t,s)}a \geq a,
\end{equation*}
then
\begin{equation*}
A_2 \leq I_3.
\end{equation*}
Now, for $I_3$ we have
\begin{eqnarray*}
\lefteqn{\int_{(\left[1-\frac{R(t,s)}{R(s,s)}\right]\frac{a}{\sigma}}^{\frac{1}{\sigma}}\int_{(a-\sigma y)\frac{R(s,s)}{R(t,s)}}^{\frac{R(s,s)}{R(t,s)}a}e^{-\frac{y^2}{2}}\frac{1}{\sqrt{V(s)}}e^{-\frac{x^2}{2V(s)}} \,\ud x \ud y }\\
&\leq& C\frac{\sigma}{\sqrt{V(s)}}e^{-\frac{\min[a^2,(a-1)^2]}{2V^*}}
\frac{R(s,s)}{R(t,s)}\int_{-\infty}^{\infty}|y|e^{-\frac{y^2}{2}}\,\ud y.
\end{eqnarray*}
Hence, we have the required upper bound for $I_3$. To conclude, note that for $I_4$ we have
\begin{eqnarray*}
I_4 &=& \int_{(\left[1-\frac{R(t,s)}{R(s,s)}\right]\frac{a}{\sigma}}^{\frac{1}{\sigma}}\int_{\frac{R(s,s)}{R(t,s)}a}^ae^{-\frac{y^2}{2}}\frac{1}{\sqrt{V(s)}}e^{-\frac{x^2}{2V(s)}} \,\ud x \ud y\\
&\leq& C\frac{1}{\sqrt{V(s)}}e^{-\frac{a^2}{2V^*}}|a|\left|1-\frac{R(s,s)}{R(t,s)}\right|.
\end{eqnarray*}
This finishes the proof of Lemma \ref{lma:crossing}.
\end{proof}


\bibliographystyle{plain}      
\bibliography{bibli}   
\end{document}